\documentclass{amsart}
\usepackage{mathrsfs,amssymb,multirow,setspace,color,amsmath}
\usepackage[a4paper, total={7in, 9in}]{geometry}
\theoremstyle{plain}
\usepackage{graphicx}
\usepackage{mathtools}

\newtheorem{theorem}{Theorem}[section]
\newtheorem{lemma}[theorem]{Lemma}

\newtheorem{corollary}[theorem]{Corollary}

\theoremstyle{definition}

\theoremstyle{remark}

\input xy
\xyoption{all}
\def\a{\mathcal{P}_E(P_i) }
\def\c{{\mathcal{P}_E(G)}}
\def\d{\mathcal{P}_E^*(G)}
\def\g{$G$ }
\def\m{\mathcal{M}(G)}
\def\G{\Gamma}

\begin{document}
\title[Certain properties of the enhanced power graph associated with a finite group]{Certain properties of the enhanced power graph associated with a finite group}



\author[Parveen, Jitender Kumar, Siddharth Singh, Xuanlong Ma]{Parveen, $\text{Jitender Kumar}^{^*}$, Siddharth Singh, Xuanlong Ma}
\email{p.parveenkumar144@gmail.com,jitenderarora09@gmail.com,sidharth$\_$0903@hotmail.com,xuanlma@mail.bnu.edu.cn}

\begin{abstract}
The enhanced power graph of a finite group $G$, denoted by $\c$, is the simple undirected graph whose vertex set is $G$ and two distinct vertices $x, y$ are adjacent if $x, y \in \langle z \rangle$ for some $z \in G$. In this article, we determine all finite groups such that the minimum degree and the vertex connectivity of $\c$ are equal. Also, we classify all groups whose (proper) enhanced power graphs are strongly regular. Further, the vertex connectivity of the enhanced power graphs associated to some nilpotent groups is obtained. Finally, we obtain a lower bound and an upper bound for the Wiener index of $\c$, where $G$ is a nilpotent group. The finite nilpotent groups attaining these bounds are also characterized.

\end{abstract}


\keywords{Enhanced power graph, nilpotent group, vertex connectivity, minimum degree, Wiener index \\ *  Corresponding author}

\maketitle

\section{Introduction}
The study of graphs related to various algebraic structures becomes important, because graphs of this type have valuable applications and are related to automata theory (see \cite{a.kelarev2009cayley,kelarev2004labelled} and the books \cite{kelarev2003graph,kelarev2002ring}). Certain graphs, viz. power graphs, commuting graphs, Cayley graphs etc., associated to groups have been studied by various researchers, see \cite{a.chakrabarty2009undirected,a.kelarev2002undirected,a.segev2001commuting}.  Segev \cite{a.segev1999finite,a.segev2001commuting}, Segev and Seitz \cite{a.segev2002anisotropic} used combinatorial parameters of certain commuting graphs to establish long standing conjectures in the theory of division algebras. A variant of commuting graphs on groups has played an important role in classification of finite simple groups, see \cite{b.aschbacher2000finite}. Hayat \emph{et al.} \cite{a.hayat2019novel} used commuting graphs associated with groups to establish some NSSD (non-singular with a singular deck) molecular graph.

In order to measure how much the power graph is close to the commuting graph of a group $G$, Aalipour \emph{et al.} \cite{a.Cameron2016} introduced a new graph called \emph{enhanced power graph.} The enhanced power graph of a group $G$ is the simple undirected graph whose vertex set is $G$ and two distinct vertices $x, y$ are adjacent if $x, y \in \langle z \rangle$ for some $z \in G$. Indeed, the enhanced power graph contains the power graph and is a spanning subgraph of the commuting graph. Aalipour \emph{et al.} \cite{a.Cameron2016} characterized the finite group $G$, for which equality holds for either two of the three graphs viz. power graph, enhanced power graph and commuting graph of $G$. Further, the enhanced power graphs have received the considerable attention by various researchers. Bera \emph{et al.} \cite{a.Bera2017} characterized the abelian groups and the non abelian $p$-groups having dominatable enhanced power graphs. Dupont \emph{et al.} \cite{a.dupont2017rainbow} determined the rainbow connection number of enhanced power graph of a finite group $G$. Later, Dupont \emph{et al.} \cite{a.dupont2017enhanced} studied the graph theoretic properties of enhanced quotient graph of a finite group $G$.  A complete description of finite groups with enhanced power graphs admitting a perfect code have been studied in \cite{a.ma2017perfect}. Ma \emph{et al.} \cite{a.ma2020metric} investigated the metric dimension of an enhanced power graph of finite groups. Hamzeh et al. \cite{a.hamzeh2017automorphism} derived the automorphism groups of enhanced power graphs of finite groups.  Zahirovi$\acute{c}$ \emph{et al.} \cite{a.zahirovic2020study} proved that two finite abelian groups are isomorphic if their enhanced power graphs are isomorphic. Also, they supplied a characterization of finite nilpotent groups whose enhanced power graphs are perfect.  Recently, Panda \emph{et al.} \cite{a.panda2021enhanced} studied the graph-theoretic properties, viz. minimum degree, independence number, matching number, strong metric dimension and perfectness, of enhanced power graphs over finite abelian groups. Moreover, the enhanced power graphs associated to non-abelian groups such as semidihedral, dihedral, dicyclic, $U_{6n}$, $V_{8n}$ etc., have been studies in \cite{a.dalal2021enhanced, a.panda2021enhanced}. Bera \emph{et al.} \cite{a.bera2021connectivity} gave an upper bound for the vertex connectivity of enhanced power graph of any finite abelian group. Moreover, they classified the finite abelian groups whose proper enhanced power graphs are connected. The connectivity of the complement of the enhanced power graph has been studied in \cite{a.ma2021note}.

In this paper, we aim to enhance the investigation of the interplay between algebraic properties of the group $G$ and its enhanced power graph $\c$. This paper is arranged as follows. In Section 2, we provide the necessary background material and fix our notations used throughout the paper. In section $3$, we classify all finite groups such that the minimum degree is equal to the vertex connectivity of $\c$. Section $4$ comprises the classification of groups whose enhanced power graphs are (strongly) regular. In Section $5$, we obtain the vertex connectivity  of $\c$, where $G$ belongs to a class of nilpotent groups. Finally, in Section $6$, we study the Wiener index of $\c$, where $G$ is a nilpotent group.

\section{Preliminaries}
 In this section, first we recall the graph theoretic notions from  \cite{b.westgraph}. A \emph{graph} $\Gamma$ is a pair  $\Gamma = (V, E)$, where $V(\Gamma)$ and $E(\Gamma)$ are the set of vertices and edges of $\Gamma$, respectively. Two distinct vertices $u_1, u_2$ are $\mathit{adjacent}$, denoted by $u_1 \sim u_2$ if there is an edge between $u_1$ and $u_2$. Otherwise, we write it as $u_1 \nsim u_2$. Let $\Gamma$ be a graph. A \emph{subgraph}  $\Gamma'$ of $\Gamma$ is the graph such that $V(\Gamma') \subseteq V(\Gamma)$ and $E(\Gamma') \subseteq E(\Gamma)$. For $X \subseteq V(\Gamma)$ the subgraph of $\Gamma$ induced by $X$ is the graph with vertex set $X$ and two vertices of $\Gamma(X)$ are adjacent if and only if they are adjacent in $\Gamma$. A graph $\Gamma$ is said to be $complete$ if every two distinct vertices are adjacent. All the vertices which are adjacent to a vertex $v \in V(\Gamma)$ is called the \emph{neighbours} of $v$.  The \emph{degree} $deg(v)$ of a vertex $v$ in a graph $\Gamma$, is the number of edges incident to $v$.  The \emph{minimum degree}, denoted by $\delta (\Gamma)$, is defined by $\delta(\Gamma)= {\rm min} \{deg(v):v\in V(\Gamma)\}$. A graph $\Gamma$ is \emph{k-regular} if the degree of every vertex in $V(\Gamma)$ is $k$. A graph $\Gamma$ is said to be \emph{strongly regular graph} with parameters $(n,k, \lambda, \mu )$ if it is  $k$-regular graph on $n$ vertices such that each pair of adjacent vertices has exactly $\lambda $ common neighbours, and each pair of non-adjacent vertices has exactly $\mu$ common neighbours. A \emph{path} in a graph is the sequence of distinct vertices with the property that each vertex in the sequence is adjacent to the next vertex of it. A graph $\Gamma$ is \emph{connected}  if each pair of vertices has a path in $\Gamma$. Otherwise, $\Gamma$ is \emph{disconnected}. The \emph{distance} between $u, v \in V(\Gamma)$, denoted by $d(u, v)$,  is the number of edges in a shortest path connecting them. For a connected graph $\Gamma$, the Wiener index $W(\G)$ is defined by 
$$W(\Gamma)=\sum\limits_{x \in V(\Gamma)} \sum\limits_{y \in V(\Gamma)}\frac{d(x,y)}{2}. $$

 The \emph{diameter} of $\Gamma$ is the maximum distance between the pair of vertices in $\Gamma$. A \emph{vertex (or edge) cut-set} in a connected graph $\Gamma$ is a set $X$ of vertices (or edges) such that the remaining subgraph $\Gamma \setminus X$, by removing the set $X$, is either disconnected or has only one vertex. The cardinality of a smallest vertex (or edge) cut-set of $\Gamma$ is called the \emph{vertex (or edge) connectivity} of $\Gamma$ and it is denoted by $\kappa(\Gamma)$ (or $\kappa'(\Gamma))$. For a connected graph $\Gamma$, it is well known that $\kappa(\Gamma) \le \kappa'(\Gamma) \le \delta(\Gamma)$. The \emph{strong product} $\Gamma _1\boxtimes \Gamma _2\boxtimes \cdot \cdot \cdot \boxtimes \Gamma _r$ of graphs $\Gamma _1,\Gamma _2,...,\Gamma _r$ is a graph such that \begin{itemize}
    \item the vertex set of $\Gamma _1\boxtimes \Gamma _2\boxtimes \cdot \cdot \cdot \boxtimes \Gamma _r$ is the Cartesian product $V(\Gamma _1) \times V(\Gamma _2)\times \cdot \cdot \cdot \times V(\Gamma _r)$; and
    \item distinct vertices $(u_1,u_2,...,u_r )$ and $(v_1,v_2,...,v_r )$ are adjacent in $\Gamma _1\boxtimes \Gamma _2\boxtimes \cdot \cdot \cdot \boxtimes \Gamma _r$ if and only if either $u_i = v_i$ or $u_i\sim v_i$ in $\Gamma _i$ for each $i\in [r]$. 
\end{itemize}

We refer the readers to \cite{b.dummit1991abstract} for basic definitions and results of group theory. A cyclic subgroup of a group $G$ is called a \emph{maximal cyclic subgroup} if it is not properly contained in any cyclic subgroup of $G$. If $G$ is a cyclic group, then $G$ is  the only maximal cyclic subgroup of $G$ is itself. The set of all maximal cyclic subgroups of $G$ is denoted by $\m$. The following result is useful for latter use.


 \begin{theorem}{\rm \cite{b.dummit1991abstract}}{\label{nilpotent}}
 Let $G$ be a finite group. Then the following statements are equivalent:
 \begin{enumerate}
     \item[(i)] $G$ is a nilpotent group.
     \item[(ii)] Every Sylow subgroup of $G$ is normal.
    \item[(iii)] $G$ is the direct product of its Sylow subgroups.
    \item[(iv)] For $x,y\in G, \  x$ and $y$ commute whenever $o(x)$ and $o(y)$ are relatively primes.
 \end{enumerate}
 \end{theorem}
All the groups considered in this paper are finite. We write  $p, p_1,p_2,...,p_r$ to be prime numbers such that $p_1 <p_2<\cdot \cdot \cdot <p_r$ and $P_i$ the unique Sylow $p_i$-subgroup of $G$ for $i\in [r]=\{1,2,...,r\}$. In view of  Theorem \ref{nilpotent}, for a nilpotent group $G$ and $x \in G$, there exists a unique element $x_i \in P_i$ such that $x=x_1x_2...x_r$, for $i\in [r]$. 
 The \emph{enhanced power graph} $\c$ of a finite group $G$ is a simple undirected graph with vertex set $G$ and two vertices are adjacent if they belong to the same cyclic subgroup of $G$. For  $X\subseteq G$, we denote by $\mathcal{P}_E(X)$ the subgraph induced by $X$.
The following results will be useful in the sequel. 
\begin{theorem}{\rm \cite[Theorem 2.4]{a.Bera2017}}{\label{complete}}
The enhanced power graph $\c$   of the group $G$  is complete if and only if $G$ is cyclic.
\end{theorem}
\begin{theorem}{\rm \cite[Theorem 3.2]{a.panda2021enhanced}} \label{minimum-degree}
For a finite group $G$, the minimum degree $\delta( \c )=m-1$, where $m$ is the order of the maximal cyclic subgroup of minimum possible order.
\end{theorem}
\begin{lemma}{\rm \cite[Lemma 2.11]{a.chattopadhyay2021minimal}}{\label{on connectivity}}
Any maximal cyclic subgroup of a finite nilpotent group $G=P_1P_2\cdots P_r$ is of the form $M_1M_2\cdots M_r $, where $M_i$ is a maximal cyclic subgroup of $P_i,  \ (1 \leq i \leq r)$.
\end{lemma}
\begin{corollary}{\label{contain}}
Let $G= P_1 P_2 \cdots P_r$ be a nilpotent group and $P_i$ is cyclic for some $i$. Then $P_i$ is contained in every maximal cyclic subgroup of $G.$
\end{corollary}

\begin{theorem}{\rm \cite[Theorem 4.2]{a.moghaddamfar2014certain}}{\label{strongly regular}}
Let $G$ be a non-trivial finite group. Then the proper power graph of $G$ is strongly regular graph if and only if G is a $p$-group of order $p^m$ for which $exp(G) = p$ or $p^m$.
\end{theorem}
\section{Equality of the Minimum Degree and the Vertex Connectivity of $\c$}
It is well known that the diameter of $\c$ is at most two. Consequently, $\kappa'(\c ) = \delta(\c )$ and so $\kappa(\c ) \le \delta(\c )$. In this section, we classify the group $G$ such that $\delta(\c )=\kappa(\c )$. We begin with the following lemma.
\begin{lemma}{\label{cut-set}}
Let \g be a non-cyclic group and $M\in \m$. Then $\overline{M}$ is a cut-set of $\c$, where $\overline{M}$ is the union of all sets of the form $M\cap \langle x \rangle$, for $x\in G\setminus M$.
\end{lemma}
\begin{proof}
Let $M=\langle a\rangle $ and $M'=\langle b \rangle$ be two maximal cyclic subgroups of \g. Then we claim that there is no path between $a$ and $b$ in $\mathcal{P}_E (G\setminus \overline{M})$. If possible, let there exists a path $a\sim x_1 \sim x_2 \sim \cdots \sim x_k \sim b$
from $a$ to $b$ in $\mathcal{P}_E (G\setminus \overline{M})$. Then $x_1\in M$. Otherwise, $\langle a,x_1\rangle$ is a cyclic subgroup of $G$ does not contained in $M$, which is not possible. We may now suppose that $x_1,x_2, \ldots, x_{r-1}\in M$ and $x_r\notin M$ for some $r\in [k]\setminus \{1\}$. Note that such $r$ exists because $x_k \sim b$ and if $x_r\in M$ for each $r\in [k]$, then $x_k\in \overline{M}$ which is not possible. Now if $x_{r-1}\in M$ then by using similar argument, we obtain $x_{r-1}\in \overline{M}.$ It follows that no such path exists and so $\overline{M}$ is a cut-set.
\end{proof}
\begin{theorem}
For the group $G$, $\delta(\c )=\kappa(\c )$ if and only if one of the following holds:
\begin{enumerate}
    \item[(i)] $G$ is a cyclic group.
    \item[(ii)] $G$ is non-cyclic and it contains a maximal cyclic subgroup of order $2.$
\end{enumerate}
\end{theorem}
\begin{proof}
First suppose that $\delta(\c )=\kappa(\c )$. If \g is cyclic then we have nothing to prove. If possible, let \g be non-cyclic group and it does not have a maximal cyclic subgroup of order $2$. By Theorem \ref{minimum-degree}, $\delta(\c )=|M|-1$, where $M \in \m$ such that $|M|$ is minimum. By Lemma \ref{cut-set}, $|\overline{M}|$ is a cut-set. Note that the generators of $M$ does not belong to $\overline{M}$. Consequently, we get $$\kappa(\mathcal{P}_E(G))\leq |\overline{M}|<|M|-1=\delta(\mathcal{P}_E(G))$$
 which is not possible. Thus, \g must have a maximal cyclic subgroup of order $2$. 
 
 To prove the converse part, suppose that \g is cyclic. Then $\delta(\c )=\kappa(\c )=n-1$ [cf. Theorem \ref{complete}]. 
 If \g is non-cyclic and it has a maximal cyclic subgroup $M$ of order $2$, then by Lemma \ref{cut-set}, $\overline{M}=\{e\}$ is a cut-set. It follows that $\kappa(\c)=1$. By Theorem \ref{minimum-degree}, $\delta(\c)=|M|-1=1$ and so $\delta(\c )=\kappa(\c )$.
 \end{proof}
 
 \section{Regularity of $\c$}
The identity element of the group $G$ is adjacent to all the other elements of $G$ in $\c$. Thus, $\c$ is regular if and only if $G$ is a finite cyclic group (cf. Theorem \ref{complete}). The proper enhanced power graph $\d$ is the subgraph of $\c$ induced by $G \setminus \{e\}$. In this section, we classify the group $G$ such that $\d$ is (strongly) regular.
 \begin{theorem}{\label{regular}}
 Let $G$ be a finite group. Then $\d$ is regular if and only if one of the following holds:
 \begin{enumerate}
     \item[(i)] $G$ is a cyclic group.
     \item[(ii)] $|M_i|=|M_j|$ and $M_i\cap M_j =\{e\}$, where $M_i, M_j \in \m$.
 \end{enumerate}
 \end{theorem}
 \begin{proof}
  Suppose that $\d$ is regular. If $G$ is cyclic then there is nothing to prove. We may now suppose that $G$ is a non-cyclic group. 
  On contrary, assume that $|M_i|\neq |M_j|$ for some $i$ and $ j$. Further, assume that $x$ and $y$ are generators of $M_i$ and $M_j$, respectively. Notice that $deg(x)=|M_i|-2$ and $deg(y)=|M_j|-2$; a contradiction to the regularity of $\d$. Thus, for $M_i, \ M_j \in \m $, we have $|M_i|=|M_j|$. Moreover, if $M_i\cap M_j \neq \{e\}$ for some $i, j \ (i\neq j)$ then for $a\neq e \in M_i\cap M_j$, we obtain $deg(a)\geq |M_i \cup M_j|-2 \neq |M_i|-2= deg(x)$. Consequently, $\d$ is not regular; a contradiction.\\
  Conversely, if $G$ is a cyclic group then by Theorem \ref{complete}, $\d$ is a complete graph and so is regular. We may now suppose that $G$ is non-cyclic. If $G$ satisfy condition (ii) then note that every element of $G\setminus \{e\}$ lies in exactly one maximal cyclic subgroup of $G$. Consequently, for each $x\in G\setminus \{e\}$ we have $deg(x)=|M_i|-2$, where $M_i$ is the maximal cyclic subgroup of $G$ containing $x$. Hence, $\d$ is regular.
 \end{proof}
A strongly regular graph is always regular. However, the converse need not be true. We show that the converse is also true for $\c$ in the following theorem.
 \begin{theorem}{\label{regular strongly regular}}
 Let $G$ be a finite group. Then $\d$ is regular if and only if $\d$ is strongly regular.
 \end{theorem}
 
 \begin{proof}
 To prove the result, it is sufficient to show that if $\d$ is regular then $\d$ is strongly regular. Suppose that $\d$ is regular then $G$ must satisfy one of the conditions given in Theorem \ref{regular}. If $G$ is cyclic then being a complete graph, $\d$ is strongly regular. If $G$ satisfies condition (ii), then by the proof of Theorem \ref{regular}, for each $x\in V(\d)$ we obtain $deg(x)=m-2$, where $m$ is the order of a maximal cyclic subgroup containing $x$. For $m=2$, $\d$ is a null graph and so is strongly regular. If $m\geq 3$, then observe that in $\d$, each pair of adjacent vertices has exactly $m-3$ common neighbours and each pair of non-adjacent vertices has no common neighbour.  Hence, $\d$ is strongly regular with parameters $(n,m-2,m-3,0)$. 
 \end{proof}
In view of {\rm \cite[Theorem 28]{a.Cameron2016}} and Theorem \ref{strongly regular}, we have the following corollary.
 \begin{corollary}
 If $G$ is non-cyclic $p$-group then $\d$ is regular if and only if the exponent of $G$ is $p$.
 \end{corollary}
\begin{theorem}{\label{regular p group}}
 Let $G$ be a non-cyclic nilpotent group. Then $\d$ is regular if and only if $G$ is a $p$-group with exponent $p$.
 \end{theorem}
 \begin{proof}
Let $G$ be a non-cyclic nilpotent group of order $n=p_1^{\lambda _1}p_2^{\lambda _2}\cdot \cdot \cdot p_r^{\lambda _r}$. To prove our result it is sufficient to prove that if $\d$ is regular then $G$ is a $p$-group. If possible, let $r \ge 2$. Since $G$ is a non-cyclic group there exists a non-cyclic Sylow subgroup $P_i$. Consequently, $P_i$  has at least two maximal cyclic subgroups, namely: $M_{i}$ and $M'_{i}$. Consider the maximal cyclic subgroups $M=M_1M_2\cdots M_{i}\cdots M_r$ and $M'=M_1M_2\cdots M'_{i}\cdots M_r$ of $G$, here $M_j$ is a maximal cyclic subgroup of $P_j$ for $j\in [r]\setminus \{i\}$. By Lemma \ref{on connectivity}, we obtain that $M$ and $M'$ are maximal cyclic subgroups of $G$ such that $M\cap M' \neq \{e\}$; a contradiction of Theorem \ref{regular}. Thus, $r = 1$ and so $G$ is $p$-group.
 \end{proof}

\begin{corollary}
Let $G$ be a finite non-cyclic abelian group. Then $\d$ is regular if and only if $G$ is an elementary abelian $p$-group. 
\end{corollary}
Based on the results obtained in this section, we posed the following conjecture which we are not able to prove. \\

\textbf{Conjecture:} Let $G$ be a finite non-cyclic group. If $\d$ is regular then $G$ is a $p$-group with exponent $p$.

\section{The vertex connectivity of $\c$}
In this section, we investigate the vertex connectivity of the enhanced power graph of some nilpotent groups.  Let $G=P_1P_2\cdots P_r$ be a nilpotent group. First, we prove that the enhanced power graph of a finite nilpotent group $G$ is isomorphic to the strong product of $\a$ ($1 \le i \le r$). Using this and ascertaining a minimum cut-set, we obtain the vertex connectivity of $\c$, where $G$ is nilpotent group such that its each Sylow subgroups is cyclic except $P_{k}$ for some $k \in [r]$. For $x\in G$, define $\tau_x=\{j\in [r]: \ x_j \neq e \}.$ Note that if $\langle x\rangle \in \m$ then $\tau_x=[r]$.

\begin{lemma}{\label{first}}
Let $H=\langle x \rangle  $ and $x=\prod\limits_{i\in \tau _x} x_i$. Then $\langle x_i \rangle  \subseteq \langle x \rangle \ \textit{for all} \  i \in \tau _x$.
\end{lemma}

\begin{proof}
Consider $i_0\in \tau _x$ and $l=\prod\limits_{i\in [r] \setminus \{i_0\}}o(x_{i})$. Then $x^l=x_{i_0}^l$ [cf. Theorem \ref{nilpotent} ]. Since $l$ is co-prime to $o(x_{i_0})$, we have $\langle x^l\rangle =\langle x_{i_0}^l\rangle =\langle x_{i_0}\rangle $ and so $x_{i_0}\in \langle x^l \rangle $. Hence, $ \langle  x_{i_0} \rangle  \subseteq  \langle  x \rangle  $.
\end{proof}

\begin{lemma}{\label{generation}} Let $G$ be a nilpotent group. Then
$  \langle  x \rangle  = \langle  \prod\limits_{i\in \tau _x} x_i \rangle   =\prod\limits_{i\in \tau _x}  \langle  x_i \rangle  ,
$
where  $  \langle  x_i \rangle   \langle  x_j \rangle   = \{ab: a\in  \langle  x_i \rangle \    \textit{and} \   b \in  \langle  x_j \rangle  \}.$
\end{lemma}

\begin{proof}
Clearly, $\langle  \prod\limits_{i\in \tau _x} x_i \rangle   \subseteq \prod\limits_{i\in \tau _x}  \langle  x_i \rangle$. If $a\in \prod\limits_{i\in \tau _x}\langle x_i \rangle$. Then $a=\prod\limits_{i\in \tau _x}a_i$ such that $a_i\in \langle x_i \rangle$. Thus, $a_i=x_i^{k_i}$ for some $k_i\in \mathbb{N}$. By Lemma \ref{first}, $a_i=x^{\lambda_i k_i}$ for some $\lambda_i \in \mathbb{N}$ and so $a=x^{\sum\limits_{i\in \tau _x}\lambda _i k_i }$. Consequently, we get $a\in \langle x \rangle$. Thus, the result holds.
\end{proof}
\begin{lemma}{\label{isomorphic lemma}}
Let $G$ be a nilpotent group such that $x = \prod\limits_{i = 1}^r x_i$ and $y =  \prod\limits_{i = 1}^r y_i$.  Then $x\sim y$ in $\c$ if and only if    $x_i\sim y_i$  in $\mathcal{P}_E(P_i)$ whenever $x_i \ne y_i$.
\end{lemma}

\begin{proof}
First suppose that $x \sim y$ in $\c$. Then there exist $z \in G$ such that $x, y \in \langle z \rangle$. We may now suppose that $x_i \ne y_i$ for some $i$. By Lemma \ref{first}, $x_i \in \langle x \rangle \subseteq \langle z \rangle$. Similarly, $y_i \in \langle z \rangle$. Thus, $x_i, y_i \in P_i$ and $\langle x_i, y_i \rangle \subseteq \langle z \rangle$ follows that $\langle x_i, y_i \rangle$ is a cyclic subgroup of $P_i$. Thus, $x_i\sim y_i$  in $\mathcal{P}_E(P_i)$.  Conversely, suppose that $x_i\sim y_i$  in $\mathcal{P}_E(P_i)$ for $x_i \ne y_i$. Consider $K=\{i\in [r]: x_i\sim y_i \  \text{in} \ \a \} $. Consequently, for $i\in K$, we have $\langle x_i,y_i\rangle \subseteq \langle z_i \rangle$ for some $z_i\in P_i$. Choose $z=\prod\limits_{i\in K}z_i\cdot \prod\limits_{i\in [r]\setminus K}x_i.$ Thus by Lemma \ref{generation},
 $\langle z\rangle = \prod\limits_{i\in K}\langle z_i\rangle \cdot \prod\limits_{i\in [r]\setminus K }\langle x_i\rangle$. Consequently, $x = \prod\limits_{i\in [r]} x_i\in \langle z\rangle$ and $y = \prod\limits_{i\in [r]} y_i\in \langle z\rangle$. Hence, $x\sim y$ in $\c$.
 \end{proof}
\begin{theorem}{\label{isomorphism theorem}}
Let \g be a nilpotent group. Then $$\c \cong \mathcal{P}_E(P_1) \boxtimes \mathcal{P}_E(P_2) \boxtimes \cdots \boxtimes \mathcal{P}_E(P_r) $$ where $P_i$ is the Sylow $p_i$-subgroup of $G.$
\end{theorem}
\begin{proof}
Let $x=x_1x_2\cdots x_r \in G$. Then define $\psi: V(\c) \rightarrow V( \mathcal{P}_E(P_1) \boxtimes \mathcal{P}_E(P_2) \boxtimes \cdots \boxtimes \mathcal{P}_E(P_r))$ such that $x \longmapsto (x_1, x_2,\ldots ,x_r)$ where $x_i\in P_i$. In view of Lemma \ref{isomorphic lemma}, note that $\psi$ is a graph isomorphism.
\end{proof}
\begin{lemma}{\label{intersection lemma}}
Let \g be a non-cyclic group and $T=\bigcap\limits_{M\in \m}M$. Then $T$ is contained in every cut-set of $\c$.
\end{lemma}
\begin{proof}
Let $x\in T$ and $y(\neq x)\in G$. Since $y\in M$ for some $M \in \m$ and $x\in T$, we have $x\in M$. Consequently, $x\sim y$. It follows that $x$ is adjacent to every vertex of $\c$. Thus, $x$ must belongs to every cut-set of $\c$ and so is $T$.
\end{proof}

\begin{theorem}
Let $G=P_1P_2\cdots P_r$ be a non-cyclic nilpotent group of order $n=p_1^{\lambda _1} p_2^{\lambda  _2}...p_r^{\lambda _r}$ with $r\geq 2$. Suppose that each Sylow subgroup is cyclic except $P_k$ for some $k\in [r]$. If $P_k$ is not a generalized quaternion group, then the set $Q=P_1\cdot \cdot \cdot P_{k-1}P_{k+1}\cdot \cdot \cdot P_r$ is the only minimum cut-set of $\c$ and hence $\kappa(\c)=\frac{n}{p_k^{\lambda _k}}$. If $P_k$ is generalized quaternion group, then the set $Q'= Z(Q_{2^\alpha})P_2 \cdot \cdot \cdot P_r$ is the only minimum cut-set of $\c$ and hence $\kappa(\c)=\frac{n}{2^{\lambda _1-1}}$.
\end{theorem}
\begin{proof}
First suppose that $P_k$ is not a generalized quaternion group. By Corollary \ref{contain}, $Q$ is contained in every maximal cyclic subgroup of $G$. By Lemma \ref{intersection lemma}, $Q$ is contained in every cut-set of $\c $. Now, to prove  our result we first prove the following claim.\\
{\rm \textbf{Claim:}} Let $G=P_1 P_2\cdots P_r$ be a finite nilpotent group and  $T_i$ be a cut-set of $\a$. Then $T=P_1 \cdots P_{i-1} T_i P_{i+1} \cdots P_r$ is a cut-set of $\c$.\\
\textit{Proof of the claim:} Let $T_i$ be a cut-set of $\a$ and let $a,b \in P_i$ such that there exist no path between $a$ and $b$ in $\mathcal{P}_E(P_i \setminus T_i)$. It follows that, for the  isomorphism $\psi$  defined in the proof of Theorem \ref{isomorphism theorem}, there is no path between $\psi(a)$ and $\psi(b)$ in the subgraph induced by $V(\boxtimes _{i=1}^r \a) \setminus \psi(T)$. Consequently, there is no path between $a$ and $b$ in $\mathcal{P}_E(G \setminus T)$. Hence, $T$ is a cut-set of $\c$.\\

Now by {\rm \cite[Theorem 1]{a.bera2021connectivity}}, $\kappa(\mathcal{P}_E(P_k))=1$ and $\{e\}$ is the only cut-set of $\mathcal{P}_E(P_k)$. Thus, above claim follows that the set $Q$ is the only minimum cut-set of $\c$. Hence, $\kappa(\c)=\frac{n}{p_k^{\lambda _k}}$.\\
Now suppose that $P_k=Q_{2^\alpha}$ is a generalized quaternion group. Note that the center   $Z(Q_{2^{\alpha}})$ of $Q_{2^{\alpha}}$ is contained in every maximal cyclic subgroup of $Q_{2^{\alpha}}$. Consequently, $Q'$ is contained in every maximal cyclic subgroup of $G$ [cf. Lemma \ref{on connectivity}]. Thus, by Lemma \ref{intersection lemma}, $Q'$ is contained in every cut-set of $\c$. By claim, $Q'$ is a cut-set of $\c$. Hence, $Q'$ is the only minimum cut-set of $\c$ and so $\kappa(\c)=\frac{n}{2^{\lambda _1-1}}$.
\end{proof}

\section{The Wiener index of $\c$}

In this section, we study the Wiener index of $\c$, where $G$ is a finite nilpotent group. We obtain a lower bound and an upper bound of $W(\c)$. We also characterize all the finite nilpotent groups attaining these bounds. Define
\begin{itemize}
    \item $S_{0,i}=\{(x,x): x\in \  P_i\}$.
    \item $S_{1,i}=\{(x,y): x\sim y \ \text{in} \ \a \}$.
    \item $S_{2,i}=\{(x,y): x\nsim y \ \text{in} \ \a \}$ such that $|S_{2,i}|=m_i$.
\end{itemize}
and 
\begin{itemize}
    \item $S_0=\{(x,x): x\in \  G\}$.
    \item $S_1=\{(x,y): x\sim y \ \text{in} \  \c \}$.
    \item $S_2=\{(x,y): x\nsim y \ \text{in} \  \c \}$.
\end{itemize}
Then by the definition of Wiener index, we have 
$$W(\c)=\frac{|S_1|+2|S_2|}{2}.$$
Now we obtain the Wiener index of $\c$, where $G$ is a nilpotent group.
\begin{theorem}{\label{winer index formula}}
Let $G$ be a nilpotent group of order $n=p_1^{\lambda _1}p_2^{\lambda _2}\cdot \cdot \cdot p_r^{\lambda _r}$. Then \[W(\c)=\frac{2n^2-n-\prod\limits_{i=1}^r(p_i^{2\lambda _i}-m_i)}{2}.\] 
\end{theorem}
\begin{proof}
By Theorem \ref{isomorphism theorem}, note that $S_1=\{(x,y): \ \text{either} \ x_i=y_i \ \text{or} \ x_i\sim y_i \ \text{in}\  \a \} \setminus S_0 $. Let $m_i=|S_{2,i}|$. Then
\begin{align*}
|S_1|& =\prod\limits_{i=1}^r(|S_{1,i}|+|S_{0,i}|)-n\\
&=\prod\limits_{i=1}^r(p_i^{2\lambda _i}-m_i-p_i^{\lambda _i}+p_i^{\lambda _i})-n\\
&=\prod\limits_{i=1}^r(p_i^{2\lambda _i}-m_i)-n
 \end{align*}
 and 
 $|S_2|=n^2-|S_0|-|S_1|=n^2-\prod\limits_{i=1}^r(p_i^{2\lambda _i}-m_i).$
 Hence, $W(\c)=\frac{2n^2-n-\prod\limits_{i=1}^r(p_i^{2\lambda _i}-m_i)}{2}$.
\end{proof}

\begin{corollary}{\label{wi comparision}}
Let $G, G'$ be nilpotent groups such that $|G|=|G'|=p_1^{\lambda _1}p_2^{\lambda _2}\cdot \cdot \cdot p_r^{\lambda _r}$. If $m_i\leq m'_i$  for all  $i\in [r]$, then $W(\c)\leq W(\mathcal{P}_E(G').$
\end{corollary}
\begin{lemma}{\label{mi}}
Let $G$ be a $p$-group. Then $|S_2|\leq (o(G)-p)(o(G)-1)$.
\end{lemma}
\begin{proof}
Let $x\neq e \in G$. Since $G$ is a $p$-group, we have $o(x)\geq p$. The number of elements at distance $2$ from $x$ is atmost $o(G)-p$. Since the identity element is adjacent to all other vertices in $\c$ and $S_2=\{(x,y): x\nsim y \ \text{in} \ \c \}$, we have $|S_2|\leq (o(G)-p)(o(G)-1).$
\end{proof}
In view of Theorem \ref{winer index formula} and Lemma \ref{mi}, we have the following corollary.
\begin{corollary}{\label{wi cor}}
Let $G$ be a $p$- group. Then $W(\c)\leq \frac{(o(G)-1)(2o(G)-p)}{2}.$ 
\end{corollary}
For the nilpotent group $G$, now we give a sharp lower bound and an upper bound of $W(\c)$ (independent from $m_i$) in the following theorem.
\begin{theorem}
Let $G$ be a nilpotent group of order $n=p_1^{\lambda _1}p_2^{\lambda _2}\cdot \cdot \cdot p_r^{\lambda _r}$. Then 
\begin{itemize}
    \item[(i)] \[\frac{n(n-1)}{2}\leq W(\c)\leq \frac{2n^2-n-\prod\limits_{i=1}^r(p_i^{\lambda _i+1} +  \; p_i^{\lambda _i} \; - \; p_i)}{2}.\]
    \item[(ii)] $W(\c)$ attains its lower bound if and only if $G$ is cyclic.
    \item[(iii)] $W(\c)$ attains its upper bound if and only if $|M| = p_1p_2\cdots p_r$ for every $M\in \m$.
\end{itemize}

\end{theorem}

\begin{proof} (i)-(ii)
 From Lemma \ref{mi}, we obtain $m_i\leq (p^{\lambda_i}-p)(p^{\lambda_i}-1)$ for all $i\in [r]$. Consequently, by Theorem \ref{winer index formula} and Corollary \ref{wi comparision}, we get $W(\c)\leq \frac{2n^2-n-\prod\limits_{i=1}^r(p_i^{\lambda _i+1}+p_i^{\lambda _i}-p_i)}{2}.$ Notice that $W(\c)$ is smallest if and only if  $\c$ is complete if and only if $G$ is cyclic (cf. Theorem \ref{complete}). Since the Wiener index of the complete graph on $n$ vertices is $\frac{n(n-1)}{2}$, we obtain $\frac{n(n-1)}{2}\leq W(\c)$.

(iii) By Theorem \ref{winer index formula}, observe that $W(\c)$ is maximum if and only if $m_i$ is maximum for all $i\in [r]$. First, we prove that $m_i$ is maximum if and only if $|M'|=p_i$ for every $M'\in \mathcal{M}(P_i)$.

 For simplicity, we write here $p_i=p$ and $\lambda _i=\lambda$ so that $m_i\leq (p^{\lambda}-p)(p^\lambda -1)$. Now let $|M'|=p$ for every $M'\in \mathcal{M}(P_i)$. Then for any non-identity element $x\in P_i$, $o(x)=p$. Since $x$ is a generator of a maximal cyclic subgroup $H$ of $P_i$ note that $x \in H$ only. It follows that $x$ is adjacent to $p-1$ vertices of $\mathcal{P}_E(P_i)$. Consequently, $x$ is at distance $2$ from $p^\lambda -p$ vertices of $\mathcal{P}_E(P_i)$. Since $x$ is an arbitrary non-identity element of $P_i$, we have $m_i=(p^\lambda -1)(p^\lambda -p)$. Thus $m_i$ is maximum. Conversely, suppose that the $m_i$ is maximum. On contrary, suppose $|M'|=p^\alpha$ for some $\alpha \geq 2$ and $M'\in \mathcal{M}(P_i)$. Further, assume that $x\in P_i$. Clearly, $o(x)\geq p$. If $x\in M'$ then $x$ is adjacent to at least $p^\alpha -1$ vertices of $\mathcal{P}_E(P_i)$ and so at most $p^\lambda - p^\alpha$ vertices are at distance $2$ from $x$ in $\mathcal{P}_E(P_i)$. Similarly, If $x\in P_i\setminus M'$ then there are at most $p^\lambda -p$ elements at distance $2$ from $x$ in $\mathcal{P}_E(P_i)$. Consequently, we get $m_i\leq (p^\alpha -1)(p^\lambda -p^\alpha)+(p^\lambda -p)(p^\lambda - p^\alpha )< (p^\lambda -1)(p^\lambda -p)$; a contradiction. Hence, $m_i$ is maximum if and only if $|M'|=p_i$ for all $M'\in \mathcal{M}(P_i)$. \\
Thus, By Lemma \ref{on connectivity}, we get $W(\c)$ is maximum if and only if $|M| = p_1p_2\cdots p_r$ for every $M\in \m$.
\end{proof}

Note that the given upper bound is tight and it is attained by the group $G=\mathbb{Z}_{p_1}^{\lambda _1}\times \mathbb{Z}_{p_2}^{\lambda _2} \times \cdot \cdot \cdot \times \mathbb{Z}_{p_r}^{\lambda _r}$. Moreover, in this case, the graph $\c$ has minimum number of edges. \\

\section*{Declarations}

\textbf{Funding}: The first author gratefully acknowledge for providing financial support to CSIR  (09/719(0110)/2019-EMR-I) government of India. 

\vspace{.3cm}
\textbf{Conflicts of interest/Competing interests}: There is no conflict of interest regarding the publishing of this paper. 

\vspace{.3cm}
\textbf{Availability of data and material (data transparency)}: Not applicable.

\vspace{.3cm}
\textbf{Code availability (software application or custom code)}: Not applicable.

\vspace{1cm}

\noindent
{\bf Parveen\textsuperscript{\normalfont 1}, Jitender Kumar\textsuperscript{\normalfont 1}, Siddharth Singh\textsuperscript{\normalfont 1}, Xuanlong Ma \textsuperscript{\normalfont 2}}

\bigskip

\noindent{\bf Addresses}:

\vspace{5pt}

\noindent
\textsuperscript{\normalfont 1}Department of Mathematics, Birla Institute of Technology and Science Pilani, Pilani-333031, India.

\vspace{5pt}

\noindent	
\textsuperscript{\normalfont 2}School of Science, Xi’an Shiyou University, Xi’an 710065, China.

\bigskip


\end{document}